\documentclass[a4paper,12pt,reqno]{amsart}
\usepackage{graphicx}
\usepackage{caption}
\usepackage{subcaption}

\usepackage{color}
\definecolor{tk}{rgb}{0.7,0.1,0.2}
\definecolor{jw}{rgb}{0.18, 0.47, 0.92}

\usepackage{amsmath,amstext,amssymb,amsopn,amsthm}
\usepackage{url}

\usepackage[margin=30mm]{geometry}
\usepackage{eucal,mathrsfs,dsfont}
\usepackage{soul}
\usepackage{color}
\usepackage{hyperref}

\newtheorem{theorem}{Theorem}[section]
\newtheorem{lemma}[theorem]{Lemma}
\newtheorem{proposition}[theorem]{Proposition}

\newtheorem{corollary}[theorem]{Corollary}

\newtheorem{conjecture}[theorem]{Conjecture}
\newtheorem{problem}{Problem}

\newcommand{\R}{\mathbb{R}}
\newcommand{\N}{\mathbb{N}}
\newcommand{\calA}{\mathcal{A}}
\newcommand{\calG}{\mathcal{G}}
\newcommand{\eps}{\varepsilon}

\DeclareMathOperator{\dist}{\operatorname{dist}}

\title[On the interior Bernoulli problem for the fractional Laplacian]{On the interior Bernoulli free boundary problem for the fractional Laplacian on an interval}

\author[T. Kulczycki]{Tadeusz Kulczycki}
\author[J. Wszoła]{Jacek Wszoła}

\thanks{T. Kulczycki was supported by the National Science Centre, Poland, grant no. 2019/33/B/ST1/02494.}
\address{Faculty of Pure and Applied Mathematics, Wroc{\l}aw University of Science and Technology, Wyb. Wyspia{\'n}skiego 27, 50-370 Wroc{\l}aw, Poland.}

\email{tadeusz.kulczycki@pwr.edu.pl}
\email{ 255718@student.pwr.edu.pl}

\begin{document}

\begin{abstract}
We study the structure of solutions of the interior Bernoulli free boundary problem for $(-\Delta)^{\alpha/2}$ on an interval $D$ with parameter $\lambda > 0$. In particular, we show that there exists a constant $\lambda_{\alpha,D} > 0$ (called the Bernoulli constant) such that the problem has no solution for $\lambda \in (0,\lambda_{\alpha,D})$, at least one solution for $\lambda = \lambda_{\alpha,D}$ and at least two solutions for $\lambda > \lambda_{\alpha,D}$. We also study the interior Bernoulli problem for the fractional Laplacian for an interval with one free boundary point. We discuss the connection of the Bernoulli problem with the corresponding variational problem and present some conjectures. In particular, we show for $\alpha = 1$ that there exists solutions of the interior Bernoulli free boundary problem for $(-\Delta)^{\alpha/2}$ on an interval which are not minimizers of the corresponding variational problem.
\end{abstract}

\maketitle

\section{Introduction}

The interior Bernoulli free boundary problem for the fractional Laplacian is formulated as follows. 
\begin{problem}
\label{bernoulli_problem}
Given $\alpha \in (0,2)$, $d \in \N$, a bounded domain $D \subset \R^d$ and a constant $\lambda > 0$ we look for a continuous function $u: \R^d \to [0,1]$ and a domain $K \subset D$ of class $C^1$ satisfying
\begin{equation*}
\left\{
\begin{aligned}
(-\Delta)^{\alpha/2}u(x)&=0 &&\text{for $x \in D \setminus \overline{K}$,}\\
u(x)&=1 &&\text{for $x \in \overline{K}$,}\\
u(x)&=0 &&\text{for $x \in D^c$,}\\
D_n^{\alpha/2}u(x) &= -\lambda&& \text{for $x \in \partial K$.}
\end{aligned}
\right.
\end{equation*}
\end{problem}
Here $(-\Delta)^{\alpha/2}$ denotes the fractional Laplacian given by
$$
(-\Delta)^{\alpha/2}f(x) = \lim_{r \to 0^+} \frac{\alpha 2^{\alpha-1} \Gamma\left(\frac{d+\alpha}{2}\right)}{\pi^{d/2} \Gamma(1-\alpha/2)}\int_{\{z \in \R^d: \, |z| > r\}}\frac{f(x+z) - f(x)}{|z|^{d+\alpha}} \, dz
$$
and $D_n^{\alpha/2}$ denotes the generalized normal derivative given by
\begin{equation}
\label{generalized_normal_derivative}
D_n^{\alpha/2}u(x) = \lim_{t \to 0^+} \frac{u(x + t n(x)) - u(x)}{t^{\alpha/2}},
\end{equation}
where $n(x)$ is the outward unit normal vector to $K$ at $x$. As usual, by a domain we understand a nonempty, connected open set, by a domain of class $C^1$ we understand a domain, which boundary is locally a graph of a $C^1$ function.

When $\alpha = 2$, that is if $(-\Delta)^{\alpha/2}$ is replaced by $(-\Delta)$ and $D_n^{\alpha/2}$ is replaced by the normal derivative $$\partial_n u(x) = \lim_{t \to 0^+} \frac{u(x + t n(x)) - u(x)}{t},$$ Problem \ref{bernoulli_problem} is just the classical interior Bernoulli problem, which have been intensively studied, see e.g. \cite{AC1981, A1989, FR1997, HS1997, HS2000, CT2002, BS2009}. It arises in various nonlinear flow laws and several physical situation e.g. electrochemical machining and potential flow in fluid mechanics.

In the classical case $(\alpha = 2)$ it is well known that Problem \ref{bernoulli_problem} does not have a solution for any positive level $\lambda$. For example, when $D$ is convex, it is proved in \cite{HS2000} that there is some positive constant $\lambda_D$ such that this problem has a solution for level $\lambda$ if and only if $\lambda \ge \lambda_D$. This constant $\lambda_D$ is called a Bernoulli constant. It is also known that even if there are solutions to the problem for some $\lambda$ there is no uniqueness in general. For example, if $D$ is a ball in $\R^d$ ($d \ge 2$) there are exactly $2$ solutions to the problem for any $\lambda > \lambda_D$, while for $\lambda = \lambda_D$ the solution is unique (see e.g. \cite[Section 3]{FR1997}). A very interesting and open question is whether for general convex bounded domains $D$ the structure of the solutions to the Bernoulli problem enjoys similar features. See \cite{CT2002} for some discussion on this question in the classical case.

The main aim of this paper is to study the structure of solutions of Problem \ref{bernoulli_problem} in the simplest geometric case i.e. when $D$ is an interval. The main results of our paper are the following theorems.
\begin{theorem}\label{alpha02}
    Let $\alpha \in (0,2)$, $x_0 \in \R$, $r > 0$ and $D = (x_0-r,x_0+r)$. Then there is a constant $\lambda_{\alpha,D} > 0$ such that 
    Problem \ref{bernoulli_problem} has
    \begin{enumerate}
        \item at least two solutions for $\lambda> \lambda_{\alpha,D}$,
        \item at least one solution for $\lambda = \lambda_{\alpha,D}$,
        \item no solution for $\lambda < \lambda_{\alpha,D}$.
    \end{enumerate}
Moreover, if $(u,K)$ is any solution of Problem \ref{bernoulli_problem}, then $K$ is symmetric with respect to $x_0$. We also have $\lambda_{\alpha,D} = \lambda_{\alpha,(x_0-r,x_0+r)} = r^{-\alpha/2} \lambda_{\alpha,(-1,1)}$. 		
\end{theorem}

The constant $\lambda_{\alpha,D}$ is called the Bernoulli constant for the fractional Laplacian for a domain $D$. Next result provides 
some estimates of this constant for interval $(-1,1)$.
\begin{theorem}
\label{estimates_lambda}
For any $\alpha \in (0,2)$ we have
$$
C_{\alpha} \left( T_{\alpha} + \frac{1}{\alpha 2^{\alpha}} \right)\le \lambda_{\alpha,(-1,1)} \le  C_\alpha \left( \frac{2}{\alpha} \right)^{\alpha/2} \left( T_\alpha + \frac{1}{\alpha 2^\alpha} \left( \frac{\alpha}{2-\alpha} \right)^\alpha \right),
$$
where $C_{\alpha} = \pi^{-1} \sin(\pi \alpha/2)$, $T_{\alpha} = B(\alpha,1-\alpha/2)$.
\end{theorem} 

Bernoulli problems for fractional Laplacians have been investigated for the first time by Caffarelli, Roquejoffre and Sire in \cite{CRS2010}. Such problems are relevant in classical physical models in mediums where long range interactions are present. Bernoulli problems for  $(-\Delta)^{\alpha/2}$ have been intensively studied in recent years see e.g. \cite{SR2012, DSS2015, DSS2015b, SSS2014, EKPSS2019, FR2022}. In these papers mainly regularity of free boundary of solutions of the related variational problem was studied. This variational problem is formulated in Section \ref{conjectures}. 

One dimensional Bernoulli problems for the fractional Laplacian for $\alpha = 1$ are related to some reaction-diffusion equations, which can model the combustion of an oil slick on the ground with the temperature diffucing above ground (see \cite{CMS2012})

Problem \ref{bernoulli_problem} for $\alpha = 1$ have been already studied in \cite{JKS2022}. In that paper (in the part devoted to the inner Bernoulli problem) only existence of the related variational problem was studied, without investigating the number of solutions.
In our paper we study all solutions to Problem \ref{bernoulli_problem} and not only solutions to the variational problem. The relation between Problem \ref{bernoulli_problem} and the variational problem is discussed in Section \ref{conjectures}.

The main difficulty in proving Theorem \ref{alpha02} is caused by nonlocality of the fractional Laplacian. From technical point of view this is manifested by the fact that there is no explicit formula of the Poisson kernel corresponding to $(-\Delta)^{\alpha/2}$ for an open set, which is the sum of 2 disjoint intervals, although the explicit formula of the Possion kernel corresponding to $(-\Delta)^{\alpha/2}$ for an interval is known. The nonlocal Problem \ref{bernoulli_problem} for an interval can be transformed to the local one in one dimension higher (see e.g. \cite{CS2007}) but this does not seem to allow to find the explicit formula for the Poisson kernel corresponding to $(-\Delta)^{\alpha/2}$ for the sum of 2 disjoint intervals. 

In our paper we study also a simplified version of Problem \ref{bernoulli_problem}. This is the interior Bernoulli problem for the fractional Laplacian for an interval with one free boundary point. It is formulated as follows. 
\begin{problem}
\label{bernoulli_problem_1}
Given $\alpha \in (0,2)$, $x_0 \in \R$, $r > 0$, $D = (x_0 - r, x_0 + r)$ and a given constant $\lambda > 0$ we look for a Borel measurable function $u: \R^d \to [0,1]$, continuous in $D$ and an interval $K = (a,x_0+r) \subset D$ satisfying
\begin{equation*}
\left\{
\begin{aligned}
(-\Delta)^{\alpha/2}u(x)&=0 &&\text{for $x \in D \setminus \overline{K}$,}\\
u(x)&=1 &&\text{for $x \in \overline{K} \setminus \{x_0 + r\}$,}\\
u(x)&=0 &&\text{for $x \in D^c$,}\\
D_n^{\alpha/2}u(a) &= -\lambda,&& 
\end{aligned}
\right.
\end{equation*}
where $D_n^{\alpha/2}$ is given by \eqref{generalized_normal_derivative}.
\end{problem}
Clearly, the point $a$ is the unique free boundary point for this solution.

The main result concerning Problem \ref{bernoulli_problem_1} is the following theorem.
\begin{theorem}\label{one_free_thm}
    Let $\alpha \in (0,2)$, $x_0 \in \R$, $r > 0$ and $D = (x_0-r,x_0+r)$. Then there is a constant $\mu_{\alpha,D} > 0$ such that the Problem \ref{bernoulli_problem_1} has
    \begin{enumerate}
        \item exactly two solutions for $\lambda> \mu_{\alpha,D}$,
        \item exactly one solution for $\lambda = \mu_{\alpha,D}$,
        \item no solution for $\lambda < \mu_{\alpha,D}$.
    \end{enumerate}
We also have $\mu_{\alpha,D} = \mu_{\alpha,(x_0-r,x_0+r)} = r^{-\alpha/2} \mu_{\alpha,(-1,1)}$. 		
\end{theorem}

For $\alpha = 1$ we are able to obtain explicit formulas for $\mu_{\alpha,D}$ and solutions of Problem~\ref{bernoulli_problem_1}. 
\begin{theorem}\label{one_free_thm_1}
Let $\alpha = 1$, $x_0 \in \R$, $r > 0$ and $D = (x_0-r,x_0+r)$. Then we have $\mu_{1,D} = 2\sqrt{2}/(\pi \sqrt{r})$. For $\lambda = \mu_{1,D}$ we have 1 solution of Problem \ref{bernoulli_problem_1} given by  
$$
K = (x_0,x_0+r), \quad  u(x) =  \frac{2}{\pi} \arctan\left(\frac{(x-x_0+r)^{1/2}}{(2(x_0-x))^{1/2}} \right), \, x \in D \setminus \overline{K}.
$$ 
For $\lambda > \mu_{1,D}$ we have 2 solutions $(K_1,u_1)$ and $(K_2,u_2)$ of Problem \ref{bernoulli_problem_1} given by
$$
K_i = (a_i, x_0+r), \quad u_i(x) = \frac{2}{\pi}\arctan \left( \frac{(x-x_0+r)^{1/2}(x_0+r-a_i)^{1/2}}{(a_i-x)^{1/2}(2r)^{1/2}} \right), 
$$
where $x \in D \setminus \overline{K_i}$, $a_i = x_0 + (-1)^i (1-\mu^2_{1,D}/\lambda^2)^{1/2}$ for $i=1,2$.
\end{theorem}

\begin{figure}
\centering
\begin{subfigure}{.5\textwidth}
  \centering
  \includegraphics[scale=0.53]{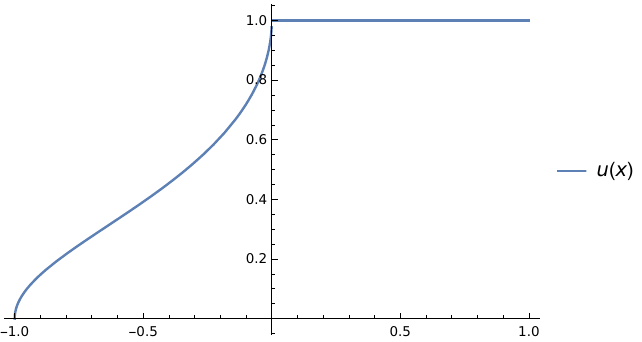}
  \caption{One solution for $\lambda = \mu_{1,D} = 2\sqrt{2}/\pi$.}
\end{subfigure}%
\begin{subfigure}{.5\textwidth}
  \centering
  \includegraphics[scale=0.53]{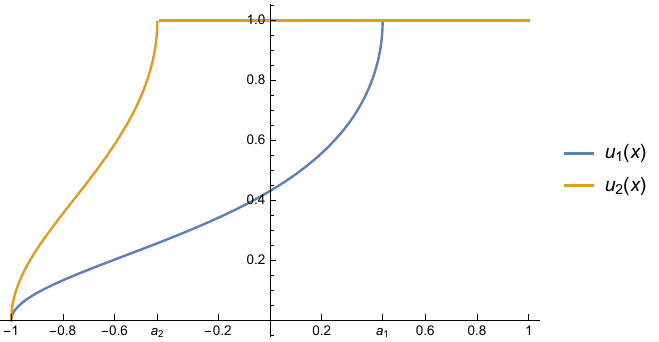}
  \caption{Two solutions for $\lambda = 1$.}
\end{subfigure}
\caption{Graphs of solutions of Problem \ref{bernoulli_problem_1} for $\alpha = 1$ and $D = (-1,1)$.} 
\label{fig:problem_2_graphs}
\end{figure}

Figure \ref{fig:problem_2_graphs} shows graphs of solutions given in Theorem \ref{one_free_thm_1}. Note, that although there is an extensive literature devoted to free boundary problems for fractional Laplacians, explicit solutions of these problems are quite rare. 

The paper is organized as follows. In Section \ref{preliminaries} we collect well known facts which will be used in the sequel. In Section \ref{one_free_point} we study the interior Bernoulli problem for the fractional Laplacian for an interval with one free boundary point. Section \ref{proof_main_results} contains proofs of main results of our paper. In Section \ref{conjectures} we present some conjectures and discuss the connection of the Bernoulli problem with the corresponding variational problem.

\section{Preliminaries}
\label{preliminaries}

In this section we present notation and gather some well known facts, which we need in the paper. In particular, we introduce Poisson kernel and Green function corresponding to the fractional Laplacian $(-\Delta)^{\alpha/2}$ (for $\alpha \in (0,2)$). We concentrate only on one-dimensional case and only on facts which are needed in this paper. For the detailed exposition of the potential theory corresponding to the fractional Laplacian we refer the reader to \cite{BBKRSV2009} or \cite{C1999}.

We denote $\N = \{1, 2, 3, \ldots\}$, for $D \subset \R$ we put $D^c = \R \setminus D$ and  for any $x \in \R$ we put $\delta_D(x) = \dist(x,D^c)$.

Let $\calG$ be a class of nonempty open sets $G$ in $\R$, $G \ne \R$ which has the representation
$$
G = \bigcup_{k = 1}^n G_k,
$$ 
where $n \in \N$ and all $G_k$ are intervals (bounded or unbounded) and $\dist(G_i,G_j) > 0$ for any $i, j \in \{1, \ldots, n\}$, $i \ne j$. The class $\calG$ plays a role of smooth sets in $\R$.

Fix $\alpha \in (0,2)$. For any $D \in \calG$ we denote by $P_D$ the Poisson kernel for $D$ corresponding to $(-\Delta)^{\alpha/2}$. $P_D: D \times (\overline{D})^c \to (0,\infty)$ has the following properties. For any $x \in D$ we have $$\int_{(\overline{D})^c} P_D(x,y) \, dy \le 1,$$ and if $D \in \calG$ is additionally bounded, then $$\int_{(\overline{D})^c} P_D(x,y) \, dy = 1.$$

Let $D \in \calG$ be bounded and let $f: D^c \to \R$ be bounded and measurable. Let us consider the following outer Dirichlet problem for the fractional Laplacian $(-\Delta)^{\alpha/2}$. We look for a bounded measurable function $u: \R \to \R$, continuous on $D$, satisfying
\begin{equation*}
\left\{
\begin{aligned}
(-\Delta)^{\alpha/2}u(x)&=0, &&\text{for $x \in D$,}\\
u(x)&=f(x), &&\text{for $x \in D^c$.}
\end{aligned}
\right.
\end{equation*}
Such Dirichlet problem has a unique solution, which is given by the formula
\begin{equation}
\label{Dirichlet_problem}
u(x) = \int_{(\overline{D})^c} P_{D}(x,y) f(y) \, dy, \quad x \in D,
\end{equation}
(see e.g. \cite[Lemma 13]{BKK2008}). When $f$ is continuous on $D^c$, it is well known that $u$ is continuous on $\R$.

On the other hand, if measurable, bounded function $u:\R \to \R$, continuous on $D$ satisfies
$$
(-\Delta)^{\alpha/2} u(x)  = 0, \quad \text{for $x \in D$}, 
$$
then the following mean value property is satisfied (see e.g. \cite[(61), (62)]{BKK2008}). For any bounded $B \subset D$, $B \in \calG$ and any $x \in B$ we have
\begin{equation}
\label{reg}
u(x) = \int_{(\overline{B})^c} P_{B}(x,y) u(y) \, dy.
\end{equation}

There are known explicit formulas of the Poisson kernels for intervals. For any $a, b \in \R$, $a < b$ we have (see e.g. \cite[(1.57)]{BBKRSV2009})
\begin{equation}
\label{Poisson_interval}
P_{(a,b)}(x,y) = C_{\alpha} \frac{((x-a)(b-x))^{\alpha/2}}{((y-a)(y-b))^{\alpha/2}} \frac{1}{|x-y|}, \quad \text{$x \in (a,b)$, $y \in [a,b]^c$,}
\end{equation}
where $C_{\alpha}$ is defined as in Theorem \ref{estimates_lambda}. 

For any $D \in \calG$ we denote by $G_D$ the Green function for $D$ corresponding to $(-\Delta)^{\alpha/2}$. It has the following properties: $G_D: \R \times \R \to [0,\infty]$ and $G_D(x,y) = 0$ if $x \in D^c$ or $y \in D^c$. For $x, y \in \R$ denote $h_{D,y}(x) = G_D(x,y)$. For any fixed $y \in D$ we have
$$
(-\Delta)^{\alpha/2} h_{D,y}(x) = 0, \quad \text{for $x \in D$.}
$$
The function $h_{D,y}$ satisfies the following mean value property (see e.g. \cite[(1.46)]{BBKRSV2009}). For any fixed $y \in D$, any bounded $B \subset D$, $B \in \calG$ such that $\dist(B,y) > 0$ we have
\begin{equation}
\label{Green_mean}
h_{D,y}(x) = \int_{(\overline{B})^c} P_{B}(x,z) h_{D,y}(z) \, dz, \quad \text{for $x \in B$.}
\end{equation}
For any $B \subset D$, $B, D \in \calG$ we have
\begin{equation}
\label{Green_comparability}
G_{B}(x,y) \le G_{D}(x,y), \quad \text{for $x, y \in \mathbb{R}$}.
\end{equation} 

As in the classical case, there is known representation of the Poisson kernel in terms of the Green function. For any bounded $D \in \calG$ we have (see e.g. \cite[(1.49)]{BBKRSV2009})
\begin{equation}
\label{Ikeda_Watanabe}
P_D(x,z) = \int_D G_D(x,y) \frac{\calA_{\alpha}}{|y - z|^{1+\alpha}} \, dy \quad \text{for $x \in D$, $z \in (\overline{D})^c$},
\end{equation}
where 
$$
\calA_{\alpha} = \frac{\alpha 2^{\alpha} \Gamma\left(\frac{1+\alpha}{2}\right)}{2 \sqrt{\pi} \Gamma\left(1 - \frac{\alpha}{2}\right)}.
$$
This is the same constant, which appear in the definition of $(-\Delta)^{\alpha/2}$ for dimension $d = 1$.

In the paper we need some well known facts concerning hypergeometric functions. For the convenience of the reader, we briefly present them below. Let $|z|<1$, $p,q,r \in \mathbb{R}$ and $-r \notin \N$. The (Gaussian) hypergeometric function is defined as $$\, _2F_1(p,q;r;z) = \sum_{n=0}^\infty \frac{(p)_n(q)_n}{(r)_n} \frac{z^n}{n!},$$ where $(\cdot)_n$ is a Pochhammer symbol. For $p \in \R$, $r > q >0$ and $|z|<1$ we have $$B(q,r-q) \, _2F_1(p,q;r;z) = \int_0^1 t^{q-1}(1-t)^{r-q-1}(1-tz)^{-p}\,dt,$$ where $B(\cdot ,\cdot)$ is the beta function. Note, that we can rewrite the above as
\begin{align}
    B(q,r-q) \, _2F_1(p,q;r;z) &= \int_1^\infty t^{p-r}(t-1)^{r-q-1}(t-z)^{-p}\,dt \label{rep1}\\ &= \int_0^\infty t^{r-q-1}(t+1)^{p-r}(t-z+1)^{-p}\,dt. \label{rep2}
\end{align}
We will also need the following easy result concerning beta function.
\begin{lemma}\label{beta}
    For $p \in (0,2)$ we have $\displaystyle B\left(p,1-\tfrac{p}{2}\right) > 1/p$.
\end{lemma}
\begin{proof}
    By definition we have $$B\left(p, 1-\tfrac{p}{2}\right) = \int_0^1 t^{p-1} (1-t)^{-p/2} \, dt > \int_0^1 t^{p-1} \, dt = \frac{1}{p}.$$
\end{proof}

\section{Bernoulli problem with one free boundary point}
\label{one_free_point} 

This section is devoted to the study of the interior Bernoulli problem for the fractional Laplacian for an interval with one free boundary point. 

For a fixed $a \in (0,1)$ let $w_a:\R \to \R$ be a (unique) measurable, bounded function, continuous on $(0,a)$, which satisfies the following Dirichlet problem:
\begin{equation*}
\left\{
\begin{aligned}
(-\Delta)^{\alpha/2}w_a(x)&=0 &&\text{for $x \in (0,a)$,}\\
w_a(x)&=0 &&\text{for $x \in [a,1)$,}\\
w_a(x)&=1 &&\text{for $x \in (0,1)^c$.}
\end{aligned}
\right.
\end{equation*}
Clearly, for each fixed $a \in (0,1)$ we have $w_a:  \R \to [0,1]$. By \eqref{Dirichlet_problem}, we have 
\begin{equation}
\label{wa_rep}
w_a(x) = \int_{[0,1]^c} P_{(0,a)}(x,y) w_a(y) \, dy, \quad \text{for $x \in (0,a)$.}
\end{equation}
By \eqref{Poisson_interval}, we get
\begin{equation}
\label{P0a}
P_{(0,a)}(x, y) = C_\alpha \left( \frac{(a-x)x}{y(y-a)} \right)^{\alpha/2} |x-y|^{-1}
\end{equation}
 for $x \in (0,a)$ and $y \in [0,a]^c$.

For $a \in (0,1)$ let us define
$$
R(a)  = \lim_{t \to 0^+} \frac{w_a(a - t) - w_a(a)}{t^{\alpha/2}}.
$$

Using (\ref{wa_rep}) and (\ref{P0a}), we obtain
    \begin{align*}
        R(a) 
        &= \lim_{t \to 0^+} \frac{1}{t^{\alpha/2}} \int_{[0,1]^c} P_{(0,a)}(a-t, z) \, dz\\
        &= \lim_{t \to 0^+} \frac{C_{\alpha}}{t^{\alpha/2}} \int_{[0,1]^c} \left( \frac{-t^2+at}{z^2-az}\right)^{\alpha/2} |a-t-z|^{-1} \, dz\\
        &=  C_{\alpha} a^{\alpha/2} \int_{[0,1]^c} (z^2-az)^{-\alpha/2} |a-z|^{-1} \, dz\\
        &= C_{\alpha} a^{\alpha/2} \left( \int_0^\infty z^{-\alpha/2}(z+a)^{-\alpha/2-1} \, dz+ \int_1^\infty z^{-\alpha/2}(z-a)^{-\alpha/2-1} \, dz \right).
    \end{align*}
By \eqref{rep2}, we get
\begin{align*}
     \int_0^\infty z^{-\alpha/2}(z+a)^{-\alpha/2-1} \, dz &= B\left(\alpha, 1-\tfrac{\alpha}{2}\right) \, _2F_1\left(\tfrac{\alpha}{2}+1, \alpha; \tfrac{\alpha}{2}+1; 1-a\right) \\ &= B\left(\alpha, 1-\tfrac{\alpha}{2}\right) a^{-\alpha}.
\end{align*}
Similarly, by \eqref{rep1}, we get
\begin{align*}
    \int_1^\infty z^{-\alpha/2}(z-a)^{-\alpha/2-1} \, dz &= B(\alpha, 1) \, _2F_1(\tfrac{\alpha}{2}+1, \alpha; \alpha+1; a) \\ 
		&= \alpha^{-1} \, _2F_1(\tfrac{\alpha}{2}+1, \alpha; \alpha+1; a).   
\end{align*}
For convenience, throughout the rest of this section we denote 
$$
F_\alpha(a) = \, _2F_1(\tfrac{\alpha}{2}+1, \alpha; \alpha+1; a).
$$ 
Recall that in the whole paper we use notation $T_{\alpha} = B(\alpha, 1-\frac{\alpha}{2})$. 

Using the above formulas we obtain
\begin{equation}\label{D(a)}
    R(a) = C_{\alpha}(T_{\alpha} a^{-\alpha/2} + a^{\alpha/2}\alpha^{-1} F_\alpha(a)).
\end{equation}

Now, fix $a \in (0,1)$ and put $u_a = 1 - w_a$, $K_a = (a,1)$. Let $n(a)$ be the outward unit normal vector to $K_a$ at $a$. Then we have
$$ 
D_n^{\alpha/2}u_a(a) = \lim_{t \to 0^+} \frac{u_a(a + t n(a)) - u_a(a)}{t^{\alpha/2}} = \lim_{t \to 0^+} \frac{u_a(a - t) - u_a(a)}{t^{\alpha/2}} = -R(a).
$$
It follows that for any fixed $a \in (0,1)$ the function $u_a$ and the set $K_a$ is a solution of Problem \ref{bernoulli_problem_1} for $D=(0,1)$ and $\lambda = R(a)$.

Now we will study properties of the function $(0,1) \ni a \mapsto R(a)$. These properties will allow to justify Theorem \ref{one_free_thm}.

\begin{proposition}\label{convexity}
    The function $(0,1) \ni a \mapsto R(a)$ is strictly convex.
\end{proposition}

\begin{proof}
By differentiating \eqref{D(a)} twice, we get
    \begin{equation*}
        \frac{d^2}{d a^2} R(a) = C_{\alpha} \left( \frac{\alpha(\alpha+2)}{4} \cdot T_{\alpha} a^{-\alpha/2-2} + 
				\alpha^{-1} \frac{d^2}{d a^2} (a^{\alpha/2}F_\alpha(a)) \right).
    \end{equation*}
Recall that
    \begin{align*}
        F_\alpha(a) = \sum_{n=0}^\infty \frac{(\tfrac{\alpha}{2}+1)_n (\alpha)_n}{(\alpha+1)_n} \frac{a^n}{n!} = \sum_{n=0}^\infty \frac{\alpha}{\alpha+n} \binom{\alpha/2+n}{n} a^n.
    \end{align*}
Hence
$$a^{\alpha/2}F_\alpha(a) = \sum_{n=0}^\infty \frac{\alpha}{\alpha+n} \binom{\alpha/2+n}{n} a^{\alpha/2+n}.$$
It follows that $\frac{d^2}{d a^2} (a^{\alpha/2}F_\alpha(a))$ is equal to
    \begin{align*}
         \frac{\alpha}{2}\left(\frac{\alpha}{2} -1\right) a^{\alpha/2-2} + \sum_{n=1}^\infty \frac{\alpha}{\alpha+n} \binom{\alpha/2+n}{n} \left( \frac{\alpha}{2}+n \right) \left( \frac{\alpha}{2} + n -1 \right) a^{\alpha/2+n-2}.
    \end{align*}
Hence
$$
\frac{d^2}{d a^2} (a^{\alpha/2}F_\alpha(a)) > \frac{\alpha}{2}\left(\frac{\alpha}{2} -1\right) a^{\alpha/2-2}.
$$
Using this and Lemma \ref{beta}, we finally obtain
    \begin{align*}
        \left(\frac{d^2}{d a^2} R(a) \right) 4a^{\alpha/2+2} C_{\alpha}^{-1} &= \alpha(\alpha+2) T_{\alpha} + \frac{4a^{\alpha/2+2}}{\alpha} \cdot \frac{d^2}{d a^2} (a^{\alpha/2}F_\alpha(a))\\ &> (\alpha+2) + (\alpha-2).
    \end{align*}
\end{proof}

\begin{proposition}\label{limits}
    We have $$\lim_{a \to 0^+} R(a) = \lim_{a \to 1^-} R(a) = \infty.$$
\end{proposition}

\begin{proof}
The limit for $a \to 0^+$ immediately follows from \eqref{D(a)}, continuity of $F_\alpha(a)$ in $a=1$ and the fact that $F_\alpha(0) = 1$.
By the inequality $(\alpha/2+1)_n > (1)_n = n!$, we get
\begin{align*}
    F_\alpha(a) = \sum_{n=0}^\infty \frac{(\alpha/2+1)_n (\alpha)_n}{(\alpha+1)_n} \frac{a^n}{n!} > \sum_{n=0}^\infty \frac{(\alpha)_n}{(\alpha+1)_n} a^n = \sum_{n=0}^{\infty} \frac{\alpha}{\alpha+n} a^n.
\end{align*}
From this inequality we obtain $$\lim_{a \to 1^-} F_\alpha(a) = \infty.$$
Therefore
$$\lim_{a \to 1^-} R(a) = C_{\alpha} \left( T_{\alpha} + \alpha^{-1} \lim_{a \to 1^-} F_\alpha(a) \right)=\infty.$$
\end{proof}

\begin{lemma}
\label{scaling_1}
Assume that $(u,K)$ is a solution of Problem \ref{bernoulli_problem_1} for $D = (-r,r)$ and $\lambda >0$, where $K = (a,r)$, $r > 0$ and $a \in D$. Let $s > 0$. Put $D_s = (-sr,sr)$, $K_s = (sa,sr)$ and define $u_s:\R \to [0,1]$ by $u_s(x) = u(x/s)$. Then $(u_s,K_s)$ is a solution of Problem \ref{bernoulli_problem_1} for $D_s$ and $s^{-\alpha/2} \lambda$.
\end{lemma}
\begin{proof}
For $x \in \overline{K_s} = s \overline{K}$ we have $x/s \in \overline{K}$, so $u_s(x) = u(x/s) = 1$. Similarly, for $x \in (D_s)^c = s D^c$ we have $x/s \in D^c$, so $u_s(x) = u(x/s) = 0$. We also have
$$
D_n^{\alpha/2} u_s(s a) = \lim_{t \to 0^+} \frac{u_s(s a - t) - u_s(s a)}{t^{\alpha/2}} =
\lim_{t \to 0^+} \frac{u\left(a - t/s\right) - u(a)}{(t/s)^{\alpha/2}} \frac{1}{s^{\alpha/2}} = -\frac{\lambda}{s^{\alpha/2}}.
$$

Put $W = D \setminus \overline{K}$. By \eqref{Dirichlet_problem}, we obtain for $x \in W$ 
\begin{equation}
\label{u_formula_1}
u(x) = \int_{(\overline{W})^c} P_W(x,z) u(z) \, dz = \int_{K} P_W(x,z) \, dz.
\end{equation}
Note that for $x \in D_s \setminus \overline{K_s} = s W$ we have $x/s \in W$. Using this, \cite[(1.62)]{BBKRSV2009} and (\ref{u_formula_1}), we obtain for $x \in D_s \setminus \overline{K_s}$ 
$$
u_s(x) = u\left(\frac{x}{s}\right) = \int_{K} P_W\left(\frac{x}{s},z\right) \, dz = \int_{K} s P_{sW}(x,sz) \, dz.
$$
By substitution $y = sz$, this is equal to
$$
\int_{K_s} P_{sW}(x,y) \, dy = \int_{(\overline{sW})^c} P_{sW}(x,y) u_s(y) \, dy.
$$
Hence $(-\Delta)^{\alpha/2} u_s(x) = 0$ for $x \in D_s \setminus \overline{K_s}$, which finishes the proof.
\end{proof}

By the definition of the fractional Laplacian one easily obtains the following result.
\begin{lemma}
\label{translation_1}
Assume that $(u,K)$ is a solution of Problem \ref{bernoulli_problem_1} for $D = (x_0-r,x_0+r)$ and $\lambda >0$, where $K = (a,x_0+r)$, $x_0 \in \R$, $r > 0$ and $a \in D$. Let $y_0 \in \R$. Put $D_* = (x_0+y_0-r, x_0+y_0+r)$, $K_* = (a+y_0,x_0+y_0)$ and define $u_*:\R \to [0,1]$ by $u_*(x) = u(x - y_0)$. Then $(u_*,K_*)$ is a solution of Problem \ref{bernoulli_problem_1} for $D_*$ and $\lambda$.
\end{lemma}

\begin{proof}[Proof of Theorem \ref{one_free_thm}]
Recall that for any fixed $a \in (0,1)$ the function $u = u_a = 1- w_a$ and $K = K_a = (a,1)$ is a solution of Problem \ref{bernoulli_problem_1} for $D=(0,1)$ and $\lambda = R(a)$. By (\ref{D(a)}), the function $(0,1) \ni a \mapsto  R(a)$ is positive and continuous. Using this and Propositions \ref{convexity}, \ref{limits}, we obtain the assertion of the theorem for $D=(0,1)$. The assertion for arbitrary $D$ follows from Lemmas \ref{scaling_1} and \ref{translation_1}.
\end{proof}

\begin{proof}[Proof of Theorem \ref{one_free_thm_1}] We first show the assertion for $D = (0,1)$. We assume that $a \in (0,1)$. Recall that we denote $u_a = 1 - w_a$. By (\ref{wa_rep}) and (\ref{P0a}), we get for $x \in (0,a)$
\begin{align}
\nonumber
u_a(x) &= \int_a^1 P_{(0,a)}(x,z) \, dz\\
\nonumber
&= \frac{1}{\pi} \int_a^1 \frac{(a-x)^{1/2} x^{1/2}}{(z-a)^{1/2}z^{1/2} (z-x)} \, dz\\
\label{ua}
&= \frac{2}{\pi} \arctan\left(\frac{x^{1/2} (1-a)^{1/2}}{(a-x)^{1/2}}\right).
\end{align}
By (\ref{D(a)}), we obtain
$$
-D_n^{1/2}u_a(a) = D_n^{1/2}w_a(a) = C_{1}(T_{1} a^{-1/2} + a^{1/2} F_1(a)).
$$
Note that we have
$$
F_1(a) = \, _2F_1(3/2,1;2;a) = \frac{2 - 2 \sqrt{1-a}}{a \sqrt{1-a}}.
$$
Hence we obtain
$$
-D_n^{1/2}u_a(a) = \frac{2}{\pi \sqrt{a} \sqrt{1-a}}.
$$
We have
$$
\frac{d}{da} \left(\frac{2}{\pi \sqrt{a} \sqrt{1-a}}\right) = \frac{2a - 1}{\pi a^{3/2} (1-a)^{3/2}}.
$$
Thus, the minimum of the function $(0,1) \ni a \mapsto -D_n^{1/2}u_a(a)$ is obtained for $a = 1/2$ and it is equal to $4/\pi$.
Therefore,
\begin{equation}\label{mu_D1}
    \mu_{1,(0,1)} = \frac{4}{\pi}.
\end{equation}
For $\lambda = \mu_{1,(0,1)}$ the unique solution is given by $K = (1/2, 1)$ and $ u = u_{1/2}$ (given by \eqref{ua}). For any $\lambda > \mu_{1,(0,1)}$ we have $-D_n^{1/2}u_a(a) = \lambda$ if and only if
$$
\frac{2}{\pi \sqrt{a} \sqrt{1-a}} = \lambda.
$$
The above can be reduced to the following quadratic equation:
$$
a^2 - a + \frac{\mu_{1,(0,1)}^2}{4 \lambda^2} = 0.
$$
Since we have chosen $\lambda > \mu_{1,(0,1)}$, it has exactly two solutions, given by 
$$
a_1 = \frac{1 + \sqrt{1 - \mu_{1,(0,1)}^2/\lambda^2}}{2}, \quad a_2 = \frac{1 - \sqrt{1 - \mu_{1,(0,1)}^2/\lambda^2}}{2}.
$$
This implies that Problem \ref{bernoulli_problem_1} has exactly two solutions. The first solution is $K = (a_1,1)$ and $u = u_{a_1}$. The second solution is $K = (a_2,1)$ and $u = u_{a_2}$. Functions $u_{a_1}$, $u_{a_2}$ are given by \eqref{ua}. This gives the assertion of the theorem for $D = (0,1)$. 

The assertion for arbitrary $D$ follows from Lemmas \ref{scaling_1} and \ref{translation_1}.
\end{proof}

\section{Proofs of main results}
\label{proof_main_results}

In this section we present proofs of Theorems \ref{alpha02} and \ref{estimates_lambda}. 

\begin{proposition}
\label{symmetry}
Let $x_0 \in \R$, $r > 0$, $D = (x_0-r,x_0+r)$ and $\lambda > 0$. Assume that $(u,K)$ is a solution of Problem \ref{bernoulli_problem} for $D$ and $\lambda$. Then $K$ is symmetric with respect to $x_0$.
\end{proposition}

Before we prove this proposition, we need some estimates of the Green function corresponding to the fractional Laplacian. 

\begin{lemma}
\label{Green_estimates_1}
Fix $0 < b < w$. Put $U = (-w,-b) \cup (b,w)$. There exists $c_* > 0$ such that for any $x, y \in (b,w)$ we have $G_U(-x,y) \le c_* \delta_U^{\alpha/2}(x)$. The constant $c_*$ depends on $\alpha$, $b$, $w$.
\end{lemma}
\begin{proof}
By (\ref{Green_mean}) and (\ref{Poisson_interval}), for any $x, y \in (b,w)$ we have 
\begin{align*}
G_U(-x,y) &= \int_b^w P_{(-w,-b)}(-x,z) G_U(z,y) \, dz\\
&= C_{\alpha} \int_b^w \frac{((-x-(-w))(-b-(-x)))^{\alpha/2}}{((z-(-w))(z-(-b)))^{\alpha/2}} \frac{G_U(z,y)}{|-x-z|} \, dz\\
&\le c \delta_{(-w,-b)}^{\alpha/2}(-x) \int_b^w G_U(z,y) \, dz\\
&\le c_* \delta_{U}^{\alpha/2}(x),
\end{align*}
where $c$ depends on $\alpha$, $b$, $w$.
\end{proof}

\begin{lemma}
\label{Green_estimates_2}
Fix $0 < b < w$ and let $c_*$ be the constant from Lemma \ref{Green_estimates_1}. Put $U = (-w,-b) \cup (b,w)$. There exist $t \in (0,(w-b)/8)$ (depending on $\alpha$, $b$, $w$, $c_*$) such that for any $x \in (b, b+t)$ and $y \in (b+3t,b+4t)$ we have 
$$
G_U(x,y) - G_U(-x,y) \ge c_* \delta_U^{\alpha/2}(x).
$$
\end{lemma}
\begin{proof} 
Let $t \in (0,(w-b)/8)$ (which will be chosen later) and assume that $x \in (b, b+t)$ and $y \in (b+3t,b+4t)$. We have
$$
\frac{\delta_U(x) \delta_U(y)}{|x - y|^2} \le \frac{4 t^2}{4 t^2} = 1.
$$
Using this (\ref{Green_comparability}) and \cite[Corollary 3.2]{BB2000}, we get
\begin{eqnarray}
\nonumber 
G_U(x,y) &\ge& G_{(b,w)}(x,y)\\
\nonumber
&\ge& c_1 \frac{\delta_U^{\alpha/2}(x) \delta_U^{\alpha/2}(y)}{|x - y|} \\
\nonumber
&\ge& c_2 \delta_U^{\alpha/2}(x) \frac{(3t)^{\alpha/2}}{2t} \\
&=& \frac{c_3 \delta_U^{\alpha/2}(x)}{t^{1-\alpha/2}},
\end{eqnarray}
where $c_1, c_2, c_3$ depends on  $\alpha$, $b$, $w$. Let $c_*$ be the constant from Lemma \ref {Green_estimates_1}. Put $$t = \min\left(\left(\frac{c_3}{2 c_*}\right)^{2/(2-\alpha)}, \frac{w-b}{8}\right).$$ Note that we have $t^{1 - \alpha/2} \le c_3/(2 c_*)$. Using this and Lemma \ref {Green_estimates_1}, we get
$$
G_U(x,y) - G_U(-x,y) \ge 
\left(\frac{c_3}{t^{1-\alpha/2}} - c_*\right) \delta_U^{\alpha/2}(x)
\ge c_* \delta_U^{\alpha/2}(x).
$$
\end{proof}

\begin{proof}[Proof of Proposition \ref{symmetry}]
On the contrary, assume that $a - (x_0 - r) \ne (x_0 + r) - b$. We may suppose that $a - (x_0 - r) < (x_0 + r) - b$. Note that this is equivalent to $a + b < 2 x_0$. Put $w_{-1} = x_0 - r$, $w_0 = (a+b)/2$, $s = a - (x_0 - r)$, $w_1 = b + s$. We have $w_1 = b + a - (x_0 - r) < x_0 + r$. We may assume that $w_0 = 0$. Then $w_{-1} = - w_1$ and $a = -b$. Put $U = (-w_1,-b) \cup (b,w_1)$. Note that 
\begin{align}
\nonumber
D_n^{\alpha/2}u(b) - D_n^{\alpha/2}u(a) 
&= \lim_{x \to b^+} \frac{u(x) - u(b)}{\delta_U^{\alpha/2}(x)} - \lim_{x \to b^+} \frac{u(-x) - u(-b)}{\delta_U^{\alpha/2}(-x)} \\
\label{normal_difference}
&= \lim_{x \to b^+} \frac{u(x) - u(-x)}{\delta_U^{\alpha/2}(x)}.
\end{align}

We have $(-\Delta)^{\alpha/2}u(x) = 0$ for $x \in U$. The function $u$ is equal to $1$ on $[-b,b]$, and it is equal to $0$ on $(-\infty,-w_1]\cup[x_0+r,\infty)$. By (\ref{reg}) and (\ref{Ikeda_Watanabe}), for $x \in U$ we have
$$
u(x) = \int_{(\overline{U})^c}P_U(x,z) u(z) \, dz = 
\int_{(\overline{U})^c}u(z) \int_U G_U(x,y) \frac{\calA_{\alpha}}{|y - z|^{1 + \alpha}} \, dy \, dz.
$$
This is equal to
$$
\int_U G_U(x,y) h(y) \, dy,
$$
where $h(y) = \int_{(\overline{U})^c}u(z) \calA_{\alpha} |y - z|^{-1 - \alpha} \, dz$, $y \in U$. For $y \in (b,w_1)$ we have
\begin{equation}
\label{hh}
h(y) - h(-y) = \calA_{\alpha} \int_{w_1}^{x_0 + r} u(z) \left(\frac{1}{|y - z|^{1 + \alpha}} - \frac{1}{|-y - z|^{1 + \alpha}}\right)\, dz > 0.
\end{equation}
Put $U_+ = \{x \in U: \, x > 0\} = (b,w_1)$. By the same arguments as in the proof of Lemma 3.3 in \cite{K2013}, we get
$$
u(x) - u(-x) = \int_{U_+} (G_U(x,y) - G_U(-x,y)) (h(y) - h(-y)) \, dy.
$$
Using this, Lemma \ref{Green_estimates_2}, (\ref{hh}) and (\ref{normal_difference}), we get $D_n^{\alpha/2}u(b) - D_n^{\alpha/2}u(a) > 0$. This contradicts the conditions of Problem \ref{bernoulli_problem} which imply that $$D_n^{\alpha/2}u(b) - D_n^{\alpha/2}u(a) = -\lambda+\lambda = 0.$$
\end{proof}

\begin{lemma}
\label{scaling}
Assume that $(u,K)$ is a solution of Problem \ref{bernoulli_problem} for $D = (-r,r)$ and $\lambda >0$, where $K = (-a,a)$, $r > 0$ and $a \in (0,r)$. Let $s > 0$. Put $D_s = (-sr,sr)$, $K_s = (-sa,sa)$ and define $u_s:\R \to [0,1]$ by $u_s(x) = u(x/s)$. Then $(u_s,K_s)$ is a solution of Problem \ref{bernoulli_problem} for $D_s$ and $s^{-\alpha/2} \lambda$.
\end{lemma}
The proof of this lemma is very similar to the proof of Lemma \ref{scaling_1} and it is omitted.

By the definition of the fractional Laplacian, one easily obtains the following result.
\begin{lemma}
\label{translation}
Assume that $(u,K)$ is a solution of Problem \ref{bernoulli_problem} for $D = (x_0-r,x_0+r)$ and $\lambda >0$, where $K = (x_0 -a, x_0+a)$, $x_0 \in \R$, $r > 0$ and $a \in (0,r)$. Let $y_0 \in \R$. Put $D_* = (x_0+y_0-r, x_0+y_0+r)$, $K_* = (x_0+y_0 -a, x_0+y_0+a)$ and define $u_*:\R \to [0,1]$ by $u_*(x) = u(x - y_0)$. Then $(u_*,K_*)$ is a solution of Problem \ref{bernoulli_problem_1} for $D_*$ and $\lambda$.
\end{lemma}

Now, we study the solution of Problem \ref{bernoulli_problem} for $D = (-1,1)$. For a fixed $a \in (0,1)$ let $f_a$ be a (unique) bounded continuous solution of the following Dirichlet problem:
\begin{equation*}
\left\{
\begin{aligned}
(-\Delta)^{\alpha/2}f_a(x)&=0 &&\text{for $x \in (-1,-a) \cup (a,1)$,}\\
f_a(x)&=1 &&\text{for $x \in [-a,a]$,}\\
f_a(x)&=0 &&\text{for $x \in (-\infty,-1] \cup [1,\infty)$.}
\end{aligned}
\right.
\end{equation*}
Clearly, we have $f_a:  \R \to [0,1]$, the function $f_a$ satisfies $f_a(-x) = f_a(x)$ for $x \in \R$. By (\ref{Dirichlet_problem}), we have $$f_a(x) = \int_{(\overline{W(a)})^c} P_{W(a)}(x,y) f_a(y) \, dy$$ for $x \in W(a)$, where $W(a) = (-1,-a) \cup (a,1)$.

By (\ref{Poisson_interval}), we get
\begin{equation}
\label{Pa1}
P_{(a,1)}(x, y) = C_\alpha \left( \frac{(1-x)(x-a)}{(y-a)(y-1)} \right)^{\alpha/2} |x-y|^{-1}
\end{equation}
 for $x \in (a,1)$ and $y \in [a,1]^c$. By (\ref{reg}), we have
\begin{equation}
\label{fa_rep}
f_a(x) = \int_{[a,1]^c} P_{(a,1)}(x,y) f_a(y) \, dy
\end{equation}
for $x \in (a,1)$. Using this and induction, we obtain for $x \in (a,1)$ 
$$
f_a(x) = \sum_{n = 1}^{\infty} f_a^{(n)}(x),
$$
where
\begin{equation}
\label{fanx}
f_a^{(1)}(x) = \int_{-a}^a P_{(a,1)}(x,y) \, dy, \quad f_a^{(n)}(x) = \int_a^1 P_{(a,1)}(x,-y) f_a^{(n-1)}(y) \, dy, 
\end{equation}
for $n \in \N$, $n \ge 2$.

\begin{lemma}
\label{fa_continuity}
The function $(0,1) \ni a \mapsto f_a(x)$ is continuous for any fixed $x \in \R$.
\end{lemma}
\begin{proof}
By formulas (\ref{Pa1}), (\ref{fanx}) and induction for any fixed $n \in \N$ and $x \in \R$, the function $(0,1) \ni a \mapsto f_a^{(n)}(x)$ is continuous. By standard properties of the Poisson kernel, we obtain that for any $\eps > 0$ there exists $\delta > 0$ such that 
$$
\sup_{a \in [\eps,1]} \sup_{x \in (a,1)} \int_a^1 P_{(a,1)}(x,-y) \, dy \le 
\sup_{x \in (\eps,1)} \int_{\eps}^1 P_{(\eps,1)}(x,-y) \, dy
 = 1 - \delta < 1.
$$
Using this and induction, for any fixed $\eps > 0$ and arbitrary $n \in \N$ we obtain
$$
\sup_{a \in [\eps,1]} \sup_{x \in (a,1)} f_a^{(n)}(x) \le (1-\delta)^{n-1}.
$$
By this and continuity of $(0,1) \ni a \mapsto f_a^{(n)}(x)$, we get the assertion of the lemma.
\end{proof}

For $a \in (0,1)$ put 
\begin{equation}
\label{Psi_def}
\Psi(a) = - \lim_{h \to 0^+} \frac{f_a(a+h) - f_a(a)}{h^{\alpha/2}}.
\end{equation}
The next two lemmas concern some properties of the function $\Psi$.
\begin{lemma}
\label{psi_properties}
For any $a \in (0,1)$ $\Psi(a)$ is well defined and we have
\begin{align}
\label{D(a)-2}
    \Psi(a) =  C_{\alpha} (1-a)^{\alpha/2} \left( \int_1^\infty \Phi(a,y) \, dy + \int_a^\infty \Phi(a,-y) \, dy - \int_a^1 \Phi(a,-y) f_a(y)\, dy \right)
\end{align}
where 
\begin{equation}
\label{phi_formula}
\Phi(a,y)=((y-a)(y-1))^{-\alpha/2}|y-a|^{-1}.
\end{equation}
For any $a \in (0,1)$ we have $\Psi(a) \in (0,\infty)$ and $\Psi$ is continuous on $(0,1)$.
\end{lemma}
\begin{proof}
We have
\begin{align*}
\Psi(a) = \lim_{h \to 0^+} h^{-\alpha/2} \left( f_a(a) - \int_{[a,1]^c} P_{(a,1)}(a+h,y) f_a(y) \, dy \right).
\end{align*}
Moreover, $f_a(a)=1=\int_{[a,1]^c} P_{(a,1)}(a+h,y) \, dy$ for $h \in (0,1-a)$. Thus,
\begin{multline*}
\Psi(a) = \lim_{h \to 0^+} h^{-\alpha/2} \bigg( \int_1^\infty P_{(a,1)}(a+h, y) \, dy + \int_a^\infty P_{(a,1)}(a+h, -y) \, dy \\
- \int_a^1 P(a+h, -y) f_a(y) \, dy \bigg).
\end{multline*}
Using (\ref{Pa1}), the right-hand side of this equation tends to the right-hand side of (\ref{D(a)-2}). This implies that $\Psi(a)$ is well defined and gives (\ref{D(a)-2}). 

The fact that $\Psi(a) \in (0,\infty)$ for any $a \in (0,1)$ easily follows from (\ref{D(a)-2}) and (\ref{phi_formula}). (\ref{D(a)-2}), (\ref{phi_formula}) and Lemma \ref{fa_continuity} imply continuity of $\Psi$ on $(0,1)$. 
\end{proof}

\begin{lemma}
\label{asymptotics_psi}
We have $\displaystyle \lim_{a \to 0^+} \Psi(a) = \lim_{a \to 1^-} \Psi(a) = \infty$.
\end{lemma}
\begin{proof}
In the whole proof we assume that $a \in (0,1)$. Using (\ref{D(a)-2}), we get 
\begin{align*}
    \Psi(a) &> C_{\alpha} (1-a)^{\alpha/2} \int_1^\infty \Phi(a,y)\, dy \\ &= C_{\alpha} (1-a)^{\alpha/2} \int_1^\infty (y-a)^{-\alpha/2-1} (y-1)^{-\alpha/2} \, dy \\ &> C_{\alpha} (1-a)^{\alpha/2} \int_1^\infty (y-a)^{-\alpha-1} \, dy \\
		&= \frac{C_{\alpha}}{\alpha(1-a)^{\alpha/2}}.
\end{align*}
The right-hand side tends to infinity as $a \to 1^-$, so as $\Psi(a)$.

In order to prove the second limit, put $g_a(x) = 1 - f_a(x)$. By (\ref{D(a)-2}), we get
\begin{equation}
\label{psi_a_est}
\Psi(a) \geq C_{\alpha} (1-a)^{\alpha/2} \int_a^1 \Phi(a,-y) g_a(y) \, dy.
\end{equation}
For $x \in (a,1)$ we have $1 = \int_{[a,1]^c} P_{(a,1)}(x,y) \, dy$. Using this and (\ref{fa_rep}), we obtain for $x \in (a,1)$
$$
g_a(x) 
= \int_{[a,1]^c} P_{(a,1)}(x,y) g_a(y) \, dy
\ge \int_1^{\infty} P_{(a,1)}(x,y) \, dy.
$$
Note that for $x \in (a,3/4)$, $y \in (1,\infty)$ we have $(1-x)^{\alpha/2} \ge (1/4)^{\alpha/2}$, $(y-a)^{-\alpha/2}(y-1)^{-\alpha/2} (y-x)^{-1} \ge y^{-1-\alpha}$. Thus, using (\ref{Pa1}), for all such $x$ we get
\begin{align}
    \nonumber
    g_a(x) &\ge \int_1^{\infty} P_{(a,1)}(x,y) \, dy \\
    \nonumber
    &\ge C_{\alpha} 2^{-\alpha} (x - a)^{\alpha/2} \int_1^{\infty} y^{-1-\alpha} \, dy \\
    \label{g-estimate2}
    &= C_{\alpha} 2^{-\alpha} \alpha^{-1} (x - a)^{\alpha/2}.
\end{align}
Additionally, for  $y \in (2a,1)$ we have 
\begin{equation}\label{estimate2}
    \left(\frac{y-a}{y+a}\right)^{\alpha/2} 
    = \left(\frac{y/a-1}{y/a+1}\right)^{\alpha/2} 
    \ge \left(\frac{y/a-y/(2a)}{y/a+y/(2a)}\right)^{\alpha/2}
    = \left(\frac{1}{3}\right)^{\alpha/2}.
\end{equation}
Using inequalities \eqref{psi_a_est}, \eqref{g-estimate2} and \eqref{estimate2}, we obtain for $a \in (0,1/4)$
\begin{align*}
    \Psi(a) &\geq C_{\alpha} (1/2)^{\alpha/2} \int_a^{3/4} \Phi(a,-y) C_{\alpha} 2^{-\alpha} \alpha^{-1} (y - a)^{\alpha/2} \, dy \\
    &= C_{\alpha}^2 2^{-3\alpha/2} \alpha^{-1} \int_a^{3/4} (y + a)^{-1-\alpha/2} (y+1)^{-\alpha/2} (y - a)^{\alpha/2} \, dy \\
    &\geq C_{\alpha}^2 2^{-3\alpha/2} \alpha^{-1} \int_{2a}^{3/4} 2^{-\alpha/2} 3^{-\alpha/2} (y + a)^{-1} \, dy \\
    &= C_{\alpha}^2 2^{-2\alpha} 3^{-\alpha/2} \alpha^{-1} (\log(3/4+a) - \log(3a)).
\end{align*}
This implies that $$\lim_{a \to 0^+} \Psi(a) = \infty.$$
\end{proof}

\begin{proof}[Proof of Theorem \ref{alpha02}] If $(u,K)$ is a solution of Problem \ref{bernoulli_problem} for $D = (x_0-r,x_0+r)$ and some $\lambda > 0$, then, by Proposition \ref{symmetry}, $K$ is symmetric with respect to $x_0$. By Lemmas \ref{scaling}, \ref{translation}, it is obvious that it is sufficient to show the theorem for $D = (-1,1)$. Using this lemma, one also obtains $\lambda_{\alpha,D} = \lambda_{\alpha,(x_0-r,x_0+r)} = r^{-\alpha/2} \lambda_{\alpha,(-1,1)}$.

So, we may assume that $D = (-1,1)$ (that is $x_0 = 0$, $r = 1$). Hence, if $(u,K)$ is a solution of Problem \ref{bernoulli_problem} for $D$ and some $\lambda > 0$, then $u = f_a$ and $K = (-a,a)$ for some $a \in (0,1)$. 

On the other hand, by Lemma \ref{psi_properties}, for any $a \in (0,1)$ the function $f_a$ is a solution of Problem \ref{bernoulli_problem} for $D$ and $\lambda = \Psi(a)$. Now, using Lemma \ref{asymptotics_psi} and the fact that the function $(0,1) \ni a \mapsto \Psi(a)$ is continuous and positive on $(0,1)$, we obtain the assertion of the theorem.
\end{proof}

\begin{proof}[Proof of Theorem \ref{estimates_lambda}]
By \eqref{rep1}, we have for any $a \in (0,1)$ 
$$
\int_1^{\infty} \Phi(a,y) \, dy = T_{\alpha} \, _2F_1\left(1+\frac{\alpha}{2}, \alpha; 1+\frac{\alpha}{2}; a\right) = T_{\alpha} (1-a)^{-\alpha},
$$
where $\Phi$ is given by (\ref{phi_formula}). Using this and Lemma \ref{psi_properties}, for any $a \in (0,1)$ we obtain
\begin{align*}
   \Psi(a) &\geq C_{\alpha} (1-a)^{\alpha/2}\left(T_{\alpha} (1-a)^{-\alpha} + \int_1^\infty \Phi(a, -z) \, dz \right) \\
    &\geq C_{\alpha} (1-a)^{\alpha/2}\left(T_{\alpha} (1-a)^{-\alpha} + \int_1^\infty (z+1)^{-\alpha-1} \, dz \right) \\
    &= C_{\alpha} (1-a)^{\alpha/2}\left(T_{\alpha} (1-a)^{-\alpha} + \frac{1}{\alpha 2^{\alpha}} \right).
\end{align*}
Put $L(a) = C_{\alpha} (1-a)^{\alpha/2}\left(T_{\alpha} (1-a)^{-\alpha} + \alpha^{-1} 2^{-\alpha} \right)$. We will now show that $L$ is increasing on $(0,1)$. Applying Lemma \ref{beta} yields
$$
    \frac{d}{da}L(a) = \frac{C_{\alpha} \alpha}{2} (1-a)^{-1-\alpha/2} \left(T_{\alpha} - \frac{1}{\alpha 2^{\alpha}}(1-a)^{\alpha} \right)> 0  
$$
for any $a \in (0,1)$. We conclude that
$$
\min_{a \in (0,1)} \Psi(a) \geq \min_{a \in (0,1)} L(a) \geq L(0) = C_{\alpha} \left(T_{\alpha} + \frac{1}{\alpha 2^\alpha} \right).
$$
On the other hand, for any $a \in (0,1)$ we have
\begin{align*}
    \Psi(a) &\leq C_{\alpha} (1-a)^{\alpha/2} \left( T_{\alpha} (1-a)^{-\alpha} + \int_a^\infty \Phi(a, -z) \, dz\right) \\
    &\leq  C_{\alpha} (1-a)^{\alpha/2} \left( T_{\alpha} (1-a)^{-\alpha} + \int_a^\infty (a+z)^{-\alpha-1} \, dz\right) \\
    & = C_{\alpha} (1-a)^{-\alpha/2} \left( T_{\alpha} + \frac{1}{\alpha 2^\alpha} \left(\frac{1}{a}-1\right)^\alpha \right).
\end{align*}
Put $U(a) = C_{\alpha} (1-a)^{-\alpha/2} \left( T_{\alpha} + \frac{1}{\alpha 2^\alpha} \left(\frac{1}{a}-1\right)^\alpha \right)$. We observe that
$$
\min_{a \in (0,1)} \Psi(a) \leq \min_{a \in (0,1)} U(a) \leq U\left(1-
\frac{\alpha}{2}\right) = C_\alpha \left( \frac{2}{\alpha} \right)^{\alpha/2} \left( T_\alpha + \frac{1}{\alpha 2^\alpha} \left( \frac{\alpha}{2-\alpha} \right)^\alpha \right),
$$
which proves the assertion of the theorem.
\end{proof}

\section{Discussion}
\label{conjectures}

In this section we discuss properties of solutions of the inner Bernoulli problem for the fractional Laplacian on intervals and balls. We also discuss the connection of the Bernoulli problem with the corresponding variational problem. 

\begin{figure}
\centering
\begin{subfigure}{.5\textwidth}
  \centering
  \includegraphics[scale=0.53]{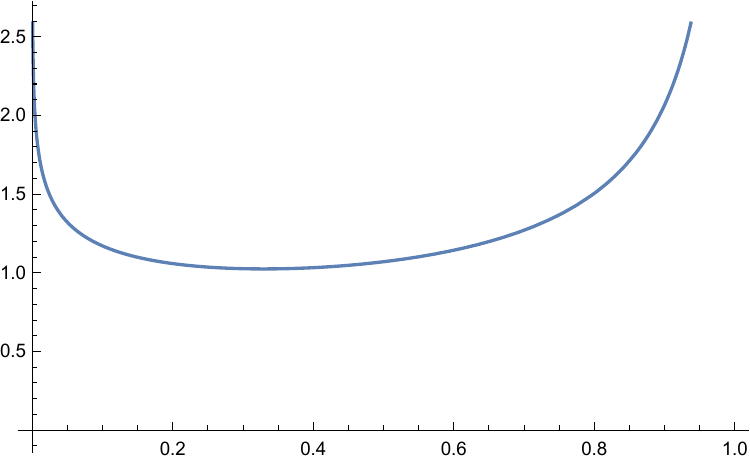}
  \caption{$\alpha = 1$}
\end{subfigure}%
\begin{subfigure}{.5\textwidth}
  \centering
  \includegraphics[scale=0.53]{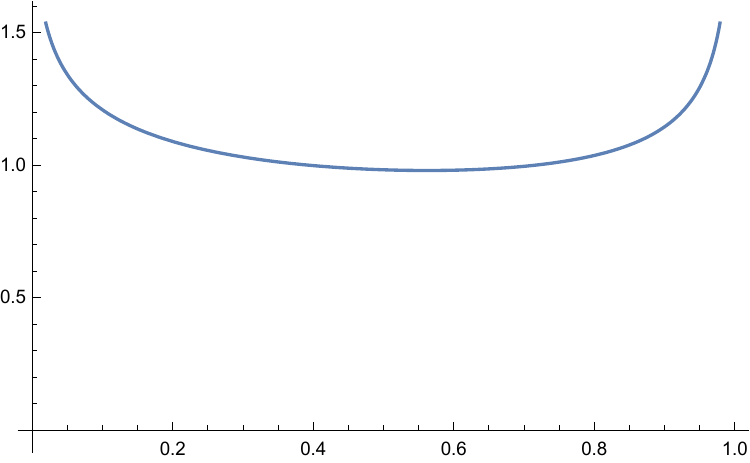}
  \caption{$\alpha = 1/2$}
\end{subfigure}
\caption{Graphs of $\Psi$ for $\alpha = 1$ and $\alpha = 1/2$.} 
\label{fig:Psigraphs}
\end{figure}

Let us start with the following conjecture.
\begin{conjecture}
\label{interval_number_of_solutions}
Let $\alpha \in (0,2)$, $x_0 \in \R$, $r > 0$ and $D = (x_0-r,x_0+r)$. Then there is a constant $\lambda_{\alpha,D} > 0$ such that the Problem \ref{bernoulli_problem} has
    \begin{enumerate}
        \item exactly two solutions for $\lambda> \lambda_{\alpha,D}$,
        \item exactly one solution for $\lambda = \lambda_{\alpha,D}$,
        \item no solution for $\lambda < \lambda_{\alpha,D}$.
    \end{enumerate}
\end{conjecture}
A standard way to show such result is to study properties of the function $\Psi$ (which is defined by (\ref{Psi_def})). In Section \ref{proof_main_results} it is shown that  $(0,1)\ni a \mapsto \Psi(a)$ is continuous and positive and $\lim_{a \to 0^+} \Psi(a) = \lim_{a \to 1^-} \Psi(a) = \infty$. So to justify Conjecture \ref{interval_number_of_solutions} it is enough to prove that for each $\alpha \in (0,2)$ functions $(0,1)\ni a \mapsto \Psi(a)$ are unimodal. Graphs of $\Psi$ for $\alpha = 1$ and $\alpha = 1/2$ (see Figure \ref{fig:Psigraphs}) suggest that these functions have this property. It seems that it is possible to justify Conjecture \ref{interval_number_of_solutions} using a computer assisted proof. Below we present some rough idea of such a proof. Of course, this is only the idea, we are not claiming in any way that this is a formal proof.

By (\ref{D(a)-2}), we have
\begin{equation}
\label{Psi_rep}
    \Psi(a) =  C_{\alpha} (1-a)^{\alpha/2} \left( \int_1^\infty \Phi(a,y) \, dy + \int_1^\infty \Phi(a,-y) \, dy + 
		\int_a^1 \Phi(a,-y) g_a(y)\, dy \right),
\end{equation}
where $g_a(y) = 1 - f_a(y)$. Fix $\alpha \in (0,2)$. It is easy to show that $\frac{d}{d a} \Psi(a)$, $\frac{d^2}{d a^2} \Psi(a)$ are well defined for any $a \in (0,1)$. Note also that for any $y \in (0,1)$ the function $(0,y) \ni a \mapsto g_{a}(y)$ is decreasing. Using this and (\ref{Psi_rep}), one can obtain that there exists $a_0$ (depending on $\alpha$) such that $\frac{d}{d a} \Psi(a) < 0$ for all $a \in (0,a_0]$. Now it is enough to show that $\Psi$ is convex on $[a_0,1)$. Note that (cf. \ref{fanx}) for $a \in (0,1)$, $x \in (a,1)$ we have $$ g_a(x) = \sum_{n = 1}^{\infty} g_a^{(n)}(x).$$
where
\begin{equation}
\label{ganx}
g_a^{(1)}(x) = \int_{[-1,1]^c} P_{(a,1)}(x,y) \, dy, \quad g_a^{(n)}(x) = \int_a^1 P_{(a,1)}(x,-y) g_a^{(n-1)}(y) \, dy, 
\end{equation}
for $n \in \N$, $n \ge 2$. For some $n_0 \in \N$ denote $\Psi = \Psi^{(1)} + \Psi^{(2)}$, where
\begin{multline*}
\Psi^{(1)}(a) =  C_{\alpha} (1-a)^{\alpha/2} \bigg( \int_1^\infty \Phi(a,y) \, dy + \int_1^\infty \Phi(a,-y) \, dy\\ 
\quad + \int_a^1 \Phi(a,-y) \sum_{k = 1}^{n_0} g_a^{(k)}(y)\, dy \bigg),
\end{multline*}
$$
\Psi^{(2)}(a) =  C_{\alpha} (1-a)^{\alpha/2} 
		\int_a^1 \Phi(a,-y) \sum_{k = n_0 + 1}^{\infty} g_a^{(k)}(y)\, dy.
$$

Now, the strategy of the proof could be the following. One could obtain that for sufficiently large $k \in \N$ $$\left|\frac{d^2}{d a^2}\left(C_{\alpha} (1-a)^{\alpha/2} \int_a^1 \Phi(a,-y) g_a^{(k)}(y)\, dy\right)\right|$$ is sufficiently small for $a \in [a_0,1)$. Then one should be able to show that there exist $\varepsilon > 0$ and $n_0 \in \N$ so that $\frac{d^2}{d a^2} \Psi^{(1)}(a) > \varepsilon$ for $a \in [a_0,1)$ and $\left|\frac{d^2}{d a^2} \Psi^{(2)}(a)\right| \le \varepsilon$. Some numerical experiments suggest that this is possible. This would show that $(0,1)\ni a \mapsto \Psi(a)$ is unimodal. Of course, the above way of proving Conjecture \ref{interval_number_of_solutions} demands numerical estimates in many steps and it is beyond the scope of this paper.  

Well known results for the classical inner Bernoulli problem for balls (see e.g. Figure 3 in \cite{FR1997}) suggest that the following hypothesis holds.
\begin{conjecture}
\label{ball_number_of_solutions}
Let $\alpha \in (0,2)$, $d \ge 2$, $x_0 \in \R^d$, $r > 0$ and $D = \{x \in \R^d: |x - x_0| <r\}$. Then there is a constant $\lambda_{\alpha,D} > 0$ such that the Problem \ref{bernoulli_problem} has
    \begin{enumerate}
        \item exactly two solutions for $\lambda> \lambda_{\alpha,D}$,
        \item exactly one solution for $\lambda = \lambda_{\alpha,D}$,
        \item no solution for $\lambda < \lambda_{\alpha,D}$.
    \end{enumerate}
\end{conjecture}
It seems, however, that this result is much more difficult to prove than Conjecture \ref{interval_number_of_solutions}. Even the proof of result similar to Theorem \ref{alpha02} for balls seems quite challenging.

One of the standard ways to study Bernoulli problems (in the classical or fractional case) is to investigate appropriate variational problems (see e.g. \cite{AC1981,CRS2010, FR2022}). Fix $d \ge 1$, $\alpha \in (0,2)$ and a bounded domain $D \subset \R^d$. Let us define the energy functional on the Sobolev space $H^{\alpha/2}(\R^{d})$
$$
i_{\alpha,\lambda,D}(u) = \frac{\alpha 2^{\alpha-2} \Gamma\left(\frac{d+\alpha}{2}\right)}{\pi^{d/2} \Gamma(1-\alpha/2)}[u]_{d,\alpha} + \left(\Gamma(1+\alpha/2)\right)^2 \lambda^{2} \, |\{x \in D: \, u(x) < 1\}|
$$
depending on the parameter $\lambda > 0$, where
$$
[u]_{d,\alpha} = \iint_{\R^d \times \R^d} \frac{(u(x) -u(y))^2}{|x - y|^{d+\alpha}} \, dx \, dy,
$$
and $|\{x \in D: \, u(x) < 1\}|$ denotes the Lebesgue measure of $\{x \in D: \, u(x) < 1\}$.
A similar definition appears in Section 1.1 in \cite{FR2022}.

Let us consider the following problem of finding minimizers to this energy functional. This problem is connected with the inner Bernoulli problem for the fractional Laplacian.
\begin{problem}
\label{variational_problem}
Given $\alpha \in (0,2)$, $\lambda>0$ and a bounded domain $D\subset \R^d$, find a nontrivial minimizer $u\in H^{\alpha/2}(\R^{d})$ of $i_{\alpha,\lambda,D}$ subject to the constraint $u=0$ on $D^c$.
\end{problem}

For $\alpha = 1$ this variational problem was studied in Section 3 in \cite{JKS2022}. In that paper the problem was investigated for general bounded domains. For the particular case when $D$ is an interval \cite{JKS2022} implies the following result.
\begin{proposition}
\label{variational_solution}
Let $x_0 \in \R$, $r > 0$ and $D = (x_0-r,x_0+r)$. There exists a constant $\Lambda_{1,D}$ such that for any $\lambda \ge \Lambda_{1,D}$ there exists a solution $u$ of Problem \ref{variational_problem} for $\alpha = 1$, $\lambda$, $D$, which is symmetric with respect to $x_0$, continuous on $\R$ and nonincreasing on $[x_0,\infty)$. The function $u$ and $K = \{x \in \R: u(x) = 1\}$ is a solution of Problem \ref{bernoulli_problem} for $\alpha = 1$, $\lambda$, $D$. For any $\lambda \in (0,\Lambda_{1,D})$ there are no solutions of Problem \ref{variational_problem} for $\alpha = 1$, $\lambda$, $D$.
\end{proposition}
Note that the above result does not guarantee the uniqueness of solution of Problem \ref{variational_problem}. The results for the classical Bernoulli problem for balls suggest that uniqueness holds (cf. Conjecture \ref{variational} below).
\begin{proof}
By translation and scaling we may assume that $x_0 = 0$ and $r = 1$. By arguments as in the proof of Lemma 2.5 in \cite{JKS2022} (by changing in that proof $E_{\lambda}$ to $I_{\lambda,(-1,1)}$), we obtain that there exists a solution of Problem \ref{variational_problem}, which is symmetric with respect to $0$, continuous on $\R$ and nonincreasing on $[0,\infty)$. It follows that $\{x \in \R: \, u(x) = 1\} = [-a,a]$ for some $a \in (0,1)$. By Theorem 1.7 (d) in \cite{JKS2022},
$$
\lim_{t \to 0^+} \frac{u(a) - u(a+t)}{\sqrt{t}} = \lambda.
$$
Using again Theorem 1.7 in \cite{JKS2022}, we obtain that $u$ and $K = \{x \in \R: u(x) = 1\}$ is a solution of Problem \ref{bernoulli_problem} for $\alpha = 1$, $\lambda$, $D = (-1,1)$.
\end{proof}

The next result gives an inequality between the variational constant $\Lambda_{1,D}$ and the Bernoulli constant $\lambda_{1,D}$ for any interval $D$. The proof of this result is computer-assisted.
\begin{proposition}
\label{inequality}
Let $x_0 \in \R$, $r > 0$ and $D = (x_0-r,x_0+r)$. We have 
\begin{equation}
\label{lambda_inequality}
\Lambda_{1,D} > \lambda_{1,D}.
\end{equation}
\end{proposition}
It is well known that the analogous result holds for the classical inner Bernoulli problem on a ball (see e.g. Example 11 in \cite{FR1997}).
\begin{proof}
In the whole proof we put $\alpha = 1$. By translation and scaling, we may assume that $x_0 = 0$ and $r = 1$ (so $D = (-1,1)$). First, we estimate $\Lambda_{1,D}$ from below. By Proposition \ref{variational_solution}, the solution $u$ of Problem \ref{variational_problem} is continuous on $\R$ and nonincreasing on $[0,\infty)$. Let $a = \max\{x\in (0,1): \, u(x) = 1\}$ and $b = \max\{x\in (0,1): \, u(x) = 1/2\}$. By arguments from the proof of Lemma 1.10 from \cite{JKS2022}, we have 
$$
\frac{\pi}{4} \Lambda_{1,D}^2 \ge \frac{\frac{1}{2\pi}}{2a} [u]_{1,1},
$$
so
\begin{align*}
\Lambda_{1,D}^2 &\ge \frac{1}{\pi^2 a} [u]_{1,1}\\
&\ge \frac{2}{\pi^2 a} \int_{-a}^a \int_{(-1,1)^c} \frac{(u(x) - u(y))^2}{|x-y|^2} \, dx \, dy\\
& \quad + \frac{2}{\pi^2 a} \int_{-a}^a \int_{(-1,-b) \cup (b,1)} \frac{(u(x) - u(y))^2}{|x-y|^2} \, dx \, dy\\
& \quad + \frac{2}{\pi^2 a} \int_{(-b,-a) \cup (a,b)} \int_{(-1,1)^c} \frac{(u(x) - u(y))^2}{|x-y|^2} \, dx \, dy\\
&= \text{I} + \text{II} + \text{III}.
\end{align*}

We have
\begin{align*}
    \text{I} &= \frac{2}{\pi^2 a} \int_{-a}^a \int_{(-1,1)^c} \frac{dx \, dy}{|x-y|^2}  =  \frac{2}{\pi^2 a} \log\left(\frac{(1+a)^2}{(1-a)^2}\right), \\
    \text{II} &= \frac{2}{\pi^2 a} 2 \int_{-a}^a \int_{b}^1 \frac{(1/4) \, dx \, dy}{|x-y|^2}  =  \frac{1}{\pi^2 a} \log\left(\frac{(1-a)(a+b)}{(1+a)(b-a)}\right), \\
    \text{III} &= \frac{2}{\pi^2 a} 2 \int_{a}^b \int_{(-1,1)^c} \frac{(1/4) \, dx \, dy}{|x-y|^2}  =  \frac{1}{\pi^2 a} \log\left(\frac{(1-a)(1+b)}{(1+a)(1-b)}\right).
\end{align*}
For $0 < a < b < 1$ put 
\begin{multline*}
   F_1(a,b) = \frac{2}{\pi^2 a} \log\left(\frac{(1+a)^2}{(1-a)^2}\right) + \frac{1}{\pi^2 a} \log\left(\frac{(1-a)(a+b)}{(1+a)(b-a)}\right) \\ + \frac{1}{\pi^2 a} \log\left(\frac{(1-a)(1+b)}{(1+a)(1-b)}\right). 
\end{multline*}
Using Mathematica, we find that
$$
\inf_{0 < a < b < 1} F_1(a,b) \ge 1.1582,
$$
so $\Lambda_{1,D}^2 \ge 1.1582$, which gives
\begin{equation}
\label{lambda_estimate}
\Lambda_{1,D} \ge 1.0761.
\end{equation}

In order to obtain upper bound estimate of $\lambda_{1,D}$, we need to estimate $\Psi(a)$. We will use formula (\ref{D(a)-2}). 
The term $f_a(y)$, present in this formula, needs to be estimated from below. We have $$f_a(y) = \sum_{n = 1}^{\infty} f_a^{(n)}(y),$$ where $f_a^{(n)}$ is given by (\ref{fanx}).

Note that for any $n \in \N$, $a \in (0,1)$ and $y \in (a,1)$ we have $f_a^{(n)}(y) > 0$. Hence $f_a(y) > f_a^{(1)}(y)$. For any $a \in (0,1)$ and $y \in (a,1)$ we have
$$
f_a^{(1)}(y) = \int_{-a}^a P_{(a,1)}(y,z) \, dz = 1 - \frac{2}{\pi} \arctan\left(\frac{\sqrt{1+a}\sqrt{y-a}}{\sqrt{2a}\sqrt{1-y}}\right).
$$
For any $a \in (0,1)$ put
$$
F_2(a) = \frac{1}{\pi} (1-a)^{\alpha/2} \left( \int_1^\infty \Phi(a,y) \, dy + \int_a^\infty \Phi(a,-y) \, dy - \int_a^1 \Phi(a,-y) f_a^{(1)}(y)\, dy \right),
$$
where $\Phi$ is given by (\ref{phi_formula}). By (\ref{D(a)-2}), we obtain $\Psi(a) \le F_2(a)$ for any $a \in (0,1)$. Using Mathematica, we find that $\min_{a \in (0,1)} F_2(a) < 1.03$. Indeed, one can check, using Mathematica, that $F_2(0.34) < 1.03$. Hence
$$
\lambda_{1,D} = \min_{a \in (0,1)} \Psi(a) < 1.03.
$$
This and (\ref{lambda_estimate}) gives (\ref{lambda_inequality}).
\end{proof}

Proposition \ref{inequality} implies the following result.
\begin{corollary}
\label{nonvariational}
For $\alpha = 1$ and any interval $D$ there exists a solution for the inner Bernoulli problem for the fractional Laplacian on $D$, which is not a minimizer of the corresponding variational problem on $D$ (Problem \ref{variational_problem}).
\end{corollary}

For any interval $D$ and $\lambda \ge \Lambda_{1,D}$ we have at least 2 solutions of Problem \ref{bernoulli_problem}. One may ask how many of them are solutions of Problem \ref{variational_problem}. Knowing results for the classical inner Bernoulli problem for balls (see e.g. Section 5.3 in \cite{FR1997}) and Propositions \ref{variational_solution}, \ref{inequality}, we may formulate the following hypothesis. 
\begin{conjecture}
\label{variational}
Let $\alpha \in (0,2)$, $x_0 \in \R$, $r > 0$ and $D = (x_0-r,x_0+r)$. There exists a constant $\Lambda_{\alpha,D}$ such that $\Lambda_{\alpha,D} > \lambda_{\alpha,D}$. For any $\lambda \ge \Lambda_{\alpha,D}$ there exists a unique solution of Problem \ref{variational_problem}. It is symmetric with respect to $x_0$, continuous on $\R$, nonincreasing on $[x_0,\infty)$ and it is a solution of Problem \ref{bernoulli_problem}. For any $\lambda \in (0,\Lambda_{\alpha,D})$ there are no solutions of Problem \ref{variational_problem}.
\end{conjecture}

\end{document}